\documentclass{article}

\usepackage{latexsym}
\usepackage{amssymb}
\usepackage{amsmath}
\usepackage{amsthm}

\newtheorem{theorem}{Theorem}
\newtheorem{lemma}[theorem]{Lemma}

\usepackage{pgf,tikz}

\usepackage{makecell}
\usepackage{relsize}

\usepackage{caption}
\usepackage{subcaption}
\usepackage{float} % provides H as float placement specifier
\usepackage[pdftex,a4paper,
citecolor = blue, colorlinks=true,urlcolor=blue]{hyperref}
\usepackage{mathrsfs}
\usetikzlibrary{arrows}

\begin{document}

\title{\bf The Topological Trees with Extreme Matula Numbers}

\author{\Large Audace A. V. Dossou-Olory\\  Department of Mathematical Sciences \\ Stellenbosch University \\ Private Bag X1, Matieland 7602 \\ South Africa\\ {\tt audace@aims.ac.za}}

\date{}

\maketitle

\begin{abstract}
Denote by $p_m$ the $m$-th prime number ($p_1=2,~p_2=3,~p_3=5,~ p_4=7,~\ldots$). Let $T$ be a rooted tree with branches $T_1,T_2,\ldots,T_r$. The Matula number $M(T)$ of $T$ is $p_{M(T_1)}\cdot p_{M(T_2)}\cdot \ldots \cdot p_{M(T_r)}$, starting with $M(K_1)=1$. This number was put forward half a century ago by the American mathematician David Matula. In this paper, we prove that the star (consisting of a root and leaves attached to it) and the binary caterpillar (a binary tree whose internal vertices form a path starting at the root) have the smallest and greatest Matula number, respectively, over all topological trees (rooted trees without vertices of outdegree $1$) with a prescribed number of leaves -- the extreme values are also derived.
\end{abstract}

\maketitle

\section{Introduction}

Fifty years ago, the American mathematician David Matula gave an explicit bijection between the set of all rooted trees and the set of all positive integers~\cite{matula1968natural}. The bijection is described by means of prime numbers. Throughout, $p_m$ always means the $m$-th prime number (in ascending order); for example,
$$
p_1=2,~p_2=3,~p_3=5,~p_4=7,~p_5=11,~p_6=13, \ldots 
$$
The Matula number of the tree $K_1$ that has only one vertex is defined to be $1$, and if $T$ is a rooted tree with branches (the components that remain after deleting the root and all edges incident to it) $T_1,T_2,\ldots,T_r$, then the Matula number of $T$ -- henceforth denoted by $M(T)$ -- is given by 
\begin{align*}
M(T)=p_{M(T_1)}\cdot p_{M(T_2)}\cdot \ldots \cdot p_{M(T_r)}\,.
\end{align*}
For example, consider the rooted tree $T$ shown in Figure~\ref{rootedTree}; $T$ has three branches $T_1,T_2,T_3$. We have
\begin{align*}
M(T_1)&=p_{M(K_1)}=p_1=2,~M(T_2)=p_{M(K_1)}\cdot p_{M(K_1)}=p_1^2=4,\\
M(T_3)&=M(K_1)=1\,,
\end{align*}
from which
\begin{align*}
M(T)=p_{M(T_1)}\cdot p_{M(T_2)} \cdot p_{M(T_3)}=p_2\cdot p_4 \cdot p_1=42
\end{align*}
is obtained.
\begin{figure}[htbp]\centering
\begin{tikzpicture}[line cap=round,line join=round,>=triangle 45,x=1.0cm,y=1.0cm, scale=1.2]
\draw [line width=1.pt] (7.,10.)-- (4.,8.);
\draw [line width=1.pt] (4.,8.)-- (4.,6.);
\draw [line width=1.pt] (7.,10.)-- (7.,8.);
\draw [line width=1.pt] (7.,8.)-- (6.,7.);
\draw [line width=1.pt] (7.,8.)-- (8.,7.);
\draw [line width=1.pt] (7.,10.)-- (9.926,6.996);
\draw [rotate around={-89.3969088056204:(4.011,6.955)},dash pattern=on 5pt off 5pt] (4.011,6.955) ellipse (1.262574606562399cm and 0.7084833358210949cm);
\draw [rotate around={-88.69804732740158:(6.967,7.128)},dash pattern=on 5pt off 5pt] (6.967,7.128) ellipse (1.3821850905018496cm and 1.2946268282426416cm);
\draw [rotate around={90.:(9.904,6.765)},dash pattern=on 5pt off 5pt] (9.904,6.765) ellipse (0.7069289531052008cm and 0.4574139752329399cm);
\draw (3.282,7.414) node[anchor=north west] {$T_1$};
\draw (6.626,7.172) node[anchor=north west] {$T_2$};
\draw (9.618,6.93) node[anchor=north west] {$T_3$};
\begin{scriptsize}
\draw [fill=black] (4.,8.) circle (2.0pt);
\draw [fill=black] (4.,6.) circle (2.0pt);
\draw [fill=black] (7.,8.) circle (2.0pt);
\draw [fill=black] (6.,7.) circle (2.0pt);
\draw [fill=black] (8.,7.) circle (2.0pt);
\draw [fill=black] (9.926,6.996) circle (2.0pt);
\draw [fill=black] (7.,10.) circle (3.0pt);
\end{scriptsize}
\end{tikzpicture}
\caption{A rooted tree $T$ with three branches $T_1,T_2,T_3$.}\label{rootedTree}
\end{figure}
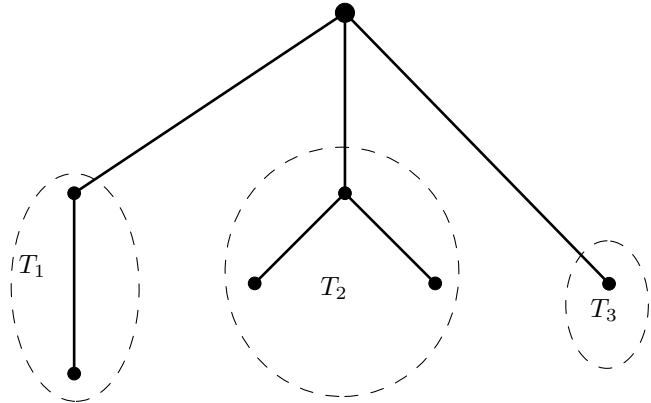

The definition of Matula numbers also suggests that for every positive integer $n$, there is a unique rooted tree $T$ whose Matula number is $n$. This observation follows by induction on $n$. Indeed, if $n>1$, then write $n$ in its (unique) prime decomposition, say $n=t_1\cdot t_2\cdot \ldots \cdot t_l$, where $t_1,t_2,\ldots,t_l$ are all primes (not necessarily distinct!). Let $m_i$ be the unique positive integer such that $t_i=p_{m_i}$; form a rooted tree $T$ by joining the roots of the trees with Matula numbers $m_1,m_2,\ldots,m_l$ to a new common vertex (the root of $T$). The Matula number of this tree is $n$. Hence, there is a bijection between the set of all rooted trees and the set of all positive integers.

\medskip
As mentioned by Ivan Gutman and Aleksandar Ivi{\'c} in their 1994 paper~\cite{gutman1994graphs}, at the time when the bijection was found by Matula, the mathematics community had considered this discovery as `just' a curiosity driven observation. One had to wait until two decades later when the American chemist Seymour Elk proved in subsequent papers~\cite{elk1995expansion,elk1994further,elk1990problem,elk1989problem} that Matula numbers can be useful for canonical nomenclatures of alkanes and all potential combinations of polybenzene or polymantane modules. In 1992, Gutman, Ivi{\'c} and Elk~\cite{IvanElk1192} obtained further results on how the molecular structures of certain organic compounds can be encoded by means of Matula numbers.

Since the Matula number also appears to be a natural (rooted) tree invariant, Gutman and Yeh listed in their 1993 paper~\cite{gutman1993deducing} ten different parameters (such as the number of vertices, number of leaves, minimum vertex degree, maximum vertex degree, diameter, etc.) of a rooted tree that can be obtained directly from the associated Matula number. Three years later, Gutman and Ivi{\'c}~\cite{gutman1996matula} proved that for $n\geq 5$, the rooted tree obtained by taking a root path (rooted at one of its endvertices) on $n-3$ vertices and attaching three leaves to the other endvertex of the path, is the one that has the maximum Matula number over the set of all rooted trees with $n$ vertices. For the minimum, they showed that for $n\geq 3$ and depending on the residue of $n$ modulo $3$, the rooted tree depicted in Figure~\ref{minVert} is minimal among all $n$-vertex rooted trees.
\begin{figure}[htbp]\centering
\begin{tikzpicture}[line cap=round,line join=round,>=triangle 45,x=1.0cm,y=1.0cm, scale=0.45]
\draw [line width=1.pt] (4.,10.)-- (3.,9.);
\draw [line width=1.pt] (3.,9.)-- (2.,8.);
\draw [line width=1.pt] (2.,8.)-- (1.,7.);
\draw [line width=1.pt] (4.,10.)-- (5.,9.);
\draw [line width=1.pt] (5.,9.)-- (6.,8.);
\draw [line width=1.pt] (6.,8.)-- (7.,7.);
\draw (0.4498945454545462,6.571049090909086) node[anchor=north west] {$1$};
\draw (5.495349090909093,6.4255945454545405) node[anchor=north west] {$\frac{n}{3} -1$};
\draw [line width=1.pt] (12.,10.)-- (11.,9.);
\draw [line width=1.pt] (11.,9.)-- (10.,8.);
\draw [line width=1.pt] (10.,8.)-- (9.,7.);
\draw [line width=1.pt] (12.,10.)-- (13.,9.);
\draw [line width=1.pt] (13.,9.)-- (14.,8.);
\draw [line width=1.pt] (14.,8.)-- (15.,7.);
\draw (8.54080363636364,6.352867272727267) node[anchor=north west] {$1$};
\draw (13.913530909090915,6.398321818181813) node[anchor=north west] {$\frac{n-1}{3}$};
\draw [line width=1.pt] (4.,10.)-- (3.,10.);
\draw [line width=1.pt] (3.,10.)-- (2.,10.);
\draw [line width=1.pt] (20.,10.)-- (19.,9.);
\draw [line width=1.pt] (19.,9.)-- (18.,8.);
\draw [line width=1.pt] (18.,8.)-- (17.,7.);
\draw [line width=1.pt] (20.,10.)-- (21.,9.);
\draw [line width=1.pt] (21.,9.)-- (22.,8.);
\draw [line width=1.pt] (22.,8.)-- (23.,7.);
\draw (20.958985454545464,6.343776363636359) node[anchor=north west] {$\frac{n-2}{3}-1$};
\draw [line width=1.pt] (20.,10.)-- (19.,10.);
\draw [line width=1.pt] (19.,10.)-- (18.,10.);
\draw [line width=1.pt] (20.,10.)-- (21.,10.);
\draw (16.66807636363637,6.33468545454545) node[anchor=north west] {$1$};
\draw [line width=1.pt] (21.,10.)-- (22.,10.);

\draw [fill=black] (4.,10.) ++(-4.5pt,0 pt) -- ++(4.5pt,4.5pt)--++(4.5pt,-4.5pt)--++(-4.5pt,-4.5pt)--++(-4.5pt,4.5pt);
\draw [fill=black] (3.,9.) circle (2.5pt);
\draw [fill=black] (2.,8.) circle (2.5pt);
\draw [fill=black] (1.,7.) circle (2.5pt);
\draw [fill=black] (5.,9.) circle (2.5pt);
\draw [fill=black] (6.,8.) circle (2.5pt);
\draw [fill=black] (7.,7.) circle (2.5pt);
\draw [fill=black] (4.,7.) circle (1.pt);
\draw [fill=black] (3.42,7.) circle (1.pt);
\draw [fill=black] (4.52,7.02) circle (1.pt);
\draw [fill=black] (12.,10.) ++(-4.5pt,0 pt) -- ++(4.5pt,4.5pt)--++(4.5pt,-4.5pt)--++(-4.5pt,-4.5pt)--++(-4.5pt,4.5pt);
\draw [fill=black] (11.,9.) circle (2.5pt);
\draw [fill=black] (10.,8.) circle (2.5pt);
\draw [fill=black] (9.,7.) circle (2.5pt);
\draw [fill=black] (13.,9.) circle (2.5pt);
\draw [fill=black] (14.,8.) circle (2.5pt);
\draw [fill=black] (15.,7.) circle (2.5pt);
\draw [fill=black] (12.,7.) circle (1.pt);
\draw [fill=black] (11.42,7.) circle (1.pt);
\draw [fill=black] (12.52,7.02) circle (1.pt);
\draw [fill=black] (3.,10.) circle (2.5pt);
\draw [fill=black] (2.,10.) circle (2.5pt);
\draw [fill=black] (20.,10.) ++(-4.5pt,0 pt) -- ++(4.5pt,4.5pt)--++(4.5pt,-4.5pt)--++(-4.5pt,-4.5pt)--++(-4.5pt,4.5pt);
\draw [fill=black] (19.,9.) circle (2.5pt);
\draw [fill=black] (18.,8.) circle (2.5pt);
\draw [fill=black] (17.,7.) circle (2.5pt);
\draw [fill=black] (21.,9.) circle (2.5pt);
\draw [fill=black] (22.,8.) circle (2.5pt);
\draw [fill=black] (23.,7.) circle (2.5pt);
\draw [fill=black] (20.,7.) circle (1.pt);
\draw [fill=black] (19.42,7.) circle (1.pt);
\draw [fill=black] (20.52,7.02) circle (1.pt);
\draw [fill=black] (19.,10.) circle (2.5pt);
\draw [fill=black] (18.,10.) circle (2.5pt);
\draw [fill=black] (21.,10.) circle (2.5pt);
\draw [fill=black] (22.,10.) circle (2.5pt);
\end{tikzpicture}
\caption{The minimal trees among all $n$-vertex rooted trees according to the residue of $n$ modulo $3$ (the root is the square vertex on top)~\cite{gutman1996matula}.}\label{minVert}
\end{figure}
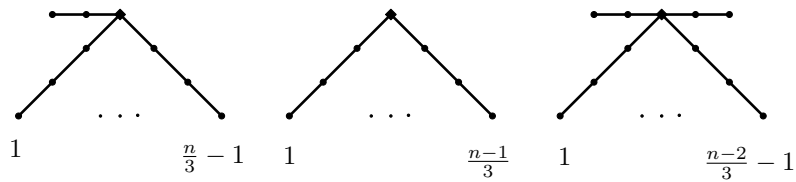

In 2012, Deutsch~\cite{deutsch2012rooted} showed how to determine further properties (mostly distance- and degree-based) of a rooted tree directly from the corresponding Matula number. As such, we have the path length, number of subtrees, Wiener index, terminal Wiener index, Wiener polynomial, first Zagreb index, second Zagreb index, just to mention a few.

In this short note, we are concerned with the Matula extremal trees (and thus the Matula extreme numbers) among the so-called topological trees, given the number of leaves.

In~\cite{bergeron1998combinatorial} a topological tree is defined as a tree without vertices of degree $2$. In this note (see also~\cite{DossouOloryWagnerPaper2}), a rooted tree will be called a \textit{topological} tree if it does not have a vertex of outdegree $1$. The purpose of avoiding vertices of outdegree $1$ in the tree is rather natural: we are collecting trees according to the number of leaves, and there are infinitely many trees with the same number of leaves (e.g. all paths; more generally, one can subdivide any set of edges in a given tree).

We mention that trees without vertices of degree $2$ (also known as \emph{series-reduced} or \emph{homeomorphically irreducible} trees) have already been object of study in the past by mathematicians and theoretical biologists; see e.g. Jamison~\cite{jamison1983average}, Bergeron et al.~\cite{bergeron1998combinatorial}, Allman and Rhodes~\cite{allman2004mathematical}, or the sequence~\url{A000669} in~\cite{oeis}.

\section{Getting to the extremal trees}\label{Towards}
Let us first give more definitions.
\begin{itemize}
\item A \textit{binary} tree is a rooted tree in which every vertex has outdegree exactly $0$ or $2$;
\item A \textit{star} is a rooted tree in which all leaves are adjacent to the root (see Figure~\ref{binaAndStars});
\item A binary \textit{caterpillar} is a binary tree whose internal vertices lie on a single path starting at the root (see Figure~\ref{binaAndStars}). 
\end{itemize}
We denote the star with $n$ leaves by $K_{1,n}$ and the $n$-leaf binary caterpillar by $F_n$. We shall prove that the star $K_{1,n}$ minimises the Matula number and the binary caterpillar $F_n$ maximises the Matula number over all topological trees with $n$ leaves. We also derive the extreme values, i.e. an expression for both $M(K_{1,n})$ and $M(F_n)$.

\begin{figure}[htbp]\centering

\begin{tikzpicture}
\draw [fill=black] (6,0) circle (2.5 pt);

\draw [fill=black] (5.5,-0.5) circle (1.5 pt);
\draw [fill=black] (6.5,-0.5) circle (1.5 pt);

\draw [fill=black] (6,-1) circle (1.5 pt);
\draw [fill=black] (7,-1) circle (1.5 pt);

\draw [fill=black] (6.5,-1.5) circle (1.5 pt);
\draw [fill=black] (7.5,-1.5) circle (1.5 pt);

\draw [fill=black] (7,-2) circle (1.5 pt);
\draw [fill=black] (8,-2) circle (1.5 pt);

%%%%%%%%%%%%

\draw [line width=1pt] (6,0) -- (5.5,-0.5);
\draw [line width=1pt] (6,0) -- (6.5,-0.5);

\draw [line width=1pt] (6.5,-0.5)  -- (6,-1);
\draw [line width=1pt] (6.5,-0.5)  -- (7,-1);

\draw [line width=1pt] (7,-1)  -- (6.5,-1.5);
\draw [line width=1pt] (7,-1)  -- (7.5,-1.5);

\draw [line width=1pt] (7.5,-1.5)  -- (7,-2);
\draw [line width=1pt] (7.5,-1.5)  -- (8,-2);

%%%%%%%%%%

\draw [fill=black] (12,0) circle (2.5 pt);

\draw [fill=black] (11.5,-2) circle (1.5 pt);
\draw [fill=black] (11,-2) circle (1.5 pt);
\draw [fill=black] (10.5,-2) circle (1.5 pt);

\draw [fill=black] (12.5,-2) circle (1.5 pt);
\draw [fill=black] (13,-2) circle (1.5 pt);
\draw [fill=black] (13.5,-2) circle (1.5 pt);

\draw [line width=1pt] (12,0) -- (11.5,-2);
\draw [line width=1pt] (12,0) -- (11.,-2);
\draw [line width=1pt] (12,0) -- (10.5,-2);

\draw [line width=1pt] (12,0) -- (12.5,-2);
\draw [line width=1pt] (12,0) -- (13.,-2);
\draw [line width=1pt] (12,0) -- (13.5,-2);
\end{tikzpicture}
\caption{The binary caterpillar $F_5$ with five leaves (left), and the star $K_{1,6}$ with six leaves (right).} \label{binaAndStars}

\end{figure}
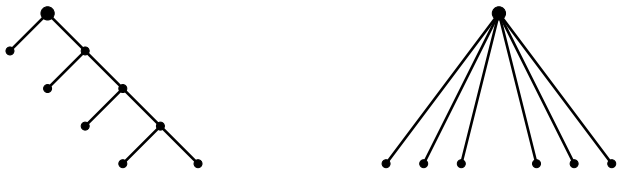

Note the following observation, which is analogous to a transformation used by the authors of~\cite{gutman1996matula} in their context of trees with a given number of vertices. Let us be given two arbitrary topological trees $T_1$ and $T_2$. Denote by $S_{1,2}$ the topological tree with the two branches $T_1$ and $T_2$. Then we have 
\begin{align*}
p_{M(T_1)} \cdot p_{M(T_2)} < p_{M(S_{1,2})}\,.
\end{align*} 
To see this, simply note that $M(S_{1,2})=p_{M(T_1)} \cdot p_{M(T_2)}$ by definition, and thus
\begin{align*}
p_{M(S_{1,2})}=p_{p_{M(T_1)} \cdot p_{M(T_2)}} > p_{M(T_1)} \cdot p_{M(T_2)}
\end{align*}
as $p_m$ (the $m$-th prime number) is greater than $m$ for all $m$.

Let $T$ be a topological tree with branches $T_1,T_2,\ldots,T_r$ such that $r\geq 3$. Again, denote by $S_{1,2}$ the topological tree whose branches are $T_1$ and $T_2$. Further, denote by $F(T)$ the rooted tree whose branches are $S_{1,2},T_3,\ldots,T_r$; see Figure~\ref{FAndFT} for a picture. Obviously, $T$ and $F(T)$ have the same number of leaves.

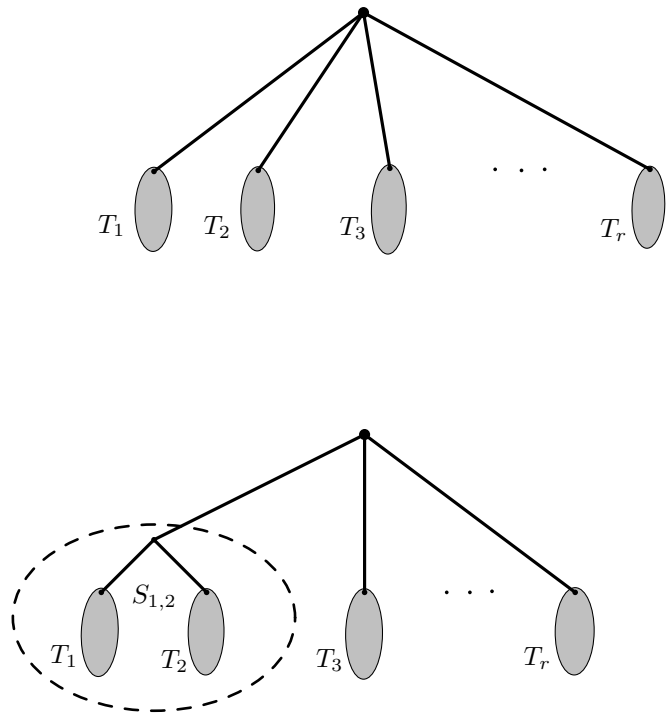
\begin{figure}[htbp]\centering
\begin{tikzpicture}[line cap=round,line join=round,>=triangle 45,x=1.0cm,y=1.0cm, scale=0.7]
	\draw (4.4,3.32) node[anchor=north west] {$S_{1,2}$};
	\draw [line width=1.2pt] (5,11)-- (8.98,14.02);
	\draw [line width=1.2pt] (8.98,14.02)-- (6.98,11.02);
	\draw [line width=1.2pt] (8.98,14.02)-- (9.48,11.06);
	\draw [line width=1.2pt] (8.98,14.02)-- (14.42,11.04);
	\draw (3.76,10.36) node[anchor=north west] {$T_1$};
	\draw (5.77,10.34) node[anchor=north west] {$T_2$};
	\draw (8.35,10.32) node[anchor=north west] {$T_3$};
	\draw (13.3,10.28) node[anchor=north west] {$T_r$};
	\draw [rotate around={89.2:(4.99,10.28)},fill=black,fill opacity=0.25] (4.99,10.28) ellipse (0.8cm and 0.35cm);
	\draw [rotate around={89.22:(6.97,10.29)},fill=black,fill opacity=0.25] (6.97,10.29) ellipse (0.8cm and 0.32cm);
	\draw [rotate around={88.53:(9.46,10.28)},fill=black,fill opacity=0.25] (9.46,10.28) ellipse (0.85cm and 0.33cm);
	\draw [rotate around={87.61:(14.39,10.32)},fill=black,fill opacity=0.25] (14.39,10.32) ellipse (0.78cm and 0.31cm);
	\draw [line width=1.2pt] (5,4)-- (9,6);
	\draw [line width=1.2pt] (9,6)-- (9,3);
	\draw [line width=1.2pt] (9,6)-- (13,3);
	
	\draw (2.86,2.18) node[anchor=north west] {$T_1$};
	\draw (4.96,2.08) node[anchor=north west] {$T_2$};
	\draw (7.9,2.1) node[anchor=north west] {$T_3$};
	\draw (11.85,2.14) node[anchor=north west] {$T_r$};
	\draw [rotate around={87.74:(3.97,2.24)},fill=black,fill opacity=0.25] (3.97,2.24) ellipse (0.84cm and 0.35cm);
	\draw [rotate around={89.24:(5.99,2.25)},fill=black,fill opacity=0.25] (5.99,2.25) ellipse (0.82cm and 0.34cm);
	\draw [line width=1.2pt] (4,3)-- (5,4);
	\draw [line width=1.2pt] (5,4)-- (6,3);
	\draw [rotate around={89.27:(8.99,2.21)},fill=black,fill opacity=0.25] (8.99,2.21) ellipse (0.86cm and 0.35cm);
	\draw [rotate around={89.23:(12.99,2.26)},fill=black,fill opacity=0.25] (12.99,2.26) ellipse (0.83cm and 0.37cm);
	\draw [rotate around={0:(5,2.52)},line width=1.pt,dash pattern=on 5pt off 5pt] (5,2.52) ellipse (2.68cm and 1.77cm);
	
	\fill [color=black] (5,11) circle (1.5pt);
	\fill [color=black] (6.98,11.02) circle (1.5pt);
	\fill [color=black] (9.48,11.06) circle (1.5pt);
	\fill [color=black] (14.42,11.04) circle (1.5pt);
	\fill [color=black] (8.98,14.02) circle (3.0pt);
	\fill [color=black] (11.48,11.04) circle (1pt);
	\fill [color=black] (11.98,11.02) circle (1pt);
	\fill [color=black] (12.42,11.04) circle (1pt);
	\fill [color=black] (5,4) circle (1.5pt);
	\fill [color=black] (4,3) circle (1.5pt);
	\fill [color=black] (9,3) circle (1.5pt);
	\fill [color=black] (13,3) circle (1.5pt);
	\fill [color=black] (9,6) circle (3.0pt);
	\fill [color=black] (10.56,3) circle (1pt);
	\fill [color=black] (11,3) circle (1pt);
	\fill [color=black] (11.42,3.02) circle (1pt);
	\fill [color=black] (6,3) circle (1.5pt);
	\end{tikzpicture}
\caption{Illustration of the tree transformation $F$: A topological tree $T$ (top figure) and the corresponding tree $F(T)$ (bottom figure).}\label{FAndFT}
\end{figure}

Then it follows immediately from the previous observation and the definition of the Matula number that 
\begin{align*}
M(F(T)) > M(T)\,.
\end{align*}
Consequently, by repeatedly applying the tree transformation $F$, we then obtain from $T$ a new topological tree, say $T^{\prime}$ with exactly two branches and having the property that $M(T^{\prime}) > M(T)$. In the same way, repeated application of the same tree transformation $F$ to the branches of $T^{\prime}$ yields a binary tree.

Hence, the tree that has the maximum Matula number among all topological trees with a prescribed number of leaves must be a binary tree. On the other hand, it also follows from our discussion that the tree that minimises the Matula number among all topological trees with $n\geq 2$ leaves must have $n$ branches. We have then proved the following theorem:

\begin{theorem}
The star $K_{1,n}$ (consisting of a root with $n$ leaves attached to it) is the topological tree that has the smallest Matula number among all $n$-leaf topological trees. Moreover, $M(K_{1,n})=2^n$ for all $n>1$.
\end{theorem}

\section{Finding the maximal topological tree}

We begin with a lemma:

\begin{lemma}\label{lemforqk}
Let the sequence $(q_k)_{k\geq 1}$ of positive integers be defined recursively by
\begin{align*}
q_1=1~\text{and}~ q_k=2 p_{q_{k-1}}~\text{for}~k>1\,.
\end{align*}
Then we have
\begin{align*}
p_{q_{k_1}} \cdot p_{q_{k_2}} \leq q_{k_1+k_2}
\end{align*}
for all pairs $(k_1,k_2)$ of positive integers.
\end{lemma}

The first few values of the sequence $(q_k)_{k\geq 1}$ are
\begin{align*}
&q_1=1,~q_2=2p_1=4,~q_3=2p_4=14,~q_4=2p_{14}=86,~q_5=2p_{86}=886,\\
&q_6=2p_{886}=13766
\end{align*}
since $p_{14}=43,~p_{86}=443$ and $p_{886}=6883$.

\medskip
To prove the lemma, we shall need the following two theorems concerning primes:

\begin{theorem}[Robin~\cite{robin1983estimation}]\label{Robin}
For any integer $m\geq 2$, we have
\begin{align*}
p_m \geq m \big(\ln(m)+\ln(\ln(m)) -1.0072629\big)\,.
\end{align*}
\end{theorem}

\begin{theorem}[Rosser-Schoenfeld~\cite{rosser1962approximate}]\label{RosSchoen}
For any integer $m\geq 20$, we have
\begin{align*}
p_m \leq m \big(\ln(m)+\ln(\ln(m)) -0.5\big)\,.
\end{align*}
\end{theorem}

\begin{proof}[Proof of Lemma~\ref{lemforqk}]
We may assume that $k_1 \leq k_2$. Set $k=k_1+k_2$. For $k=2$, we have $k_1=k_2=1$ and $p_{q_{k_1}} \cdot p_{q_{k_2}}=p_{q_1}^2=p_1^2=4=2p_1=2p_{q_1}=q_2=q_{k_1+k_2}$. For $k=3$, we have $k_1=1,~k_2=2$ and $p_{q_{k_1}} \cdot p_{q_{k_2}}=p_{q_1} \cdot p_{q_2}=p_1\cdot p_4=14=2p_4=2p_{q_2}=q_3=q_{k_1+k_2}$. So the inequality holds for $k\in\{2,3\}$. In fact, for $k\leq 6$, the inequality is easily verified; see Table~\ref{valk6} and the values of $q_1,\ldots,q_6$ given above. We then assume that $k>6$ and continue the proof of the inequality by induction on $k$.
\begin{table}[htbp]\centering
\caption{$k_1+k_2\in \{4,5, 6\}$}\label{valk6}
\begin{tabular}{c|c|c||c|c||c|c|c}
$(k_1,k_2)$                     & (1,3) & (2,2) & (1,4) & (2,3) & (1,5) & (2,4) & (3,3) \\ \hline
$p_{q_{k_1}} \cdot p_{q_{k_2}}$ & 86    & 49    & 886   & 301   & 13766 & 3101  & 1849 \\ \hline
\end{tabular}
\end{table}

Let $k_1,k_2$ be two arbitrary positive integers such that $k_1\leq k_2$ and $k_1+k_2=k$. We may assume that $k_1\geq 2$ because for $k_1=1$, the lemma clearly holds with equality. We may also assume that $k_2\geq 4$ because $k>6$ by assumption. 

Now note that $p_{q_{k_1}}\geq 4$ (since $k_1\geq 2$ implies that $q_{k_1}\geq 4$) and thus
\begin{align*}
\ln (p_{q_{k_1}}\cdot p_{q_{k_2-1}})&+\ln(\ln (p_{q_{k_1}}\cdot p_{q_{k_2-1}}))-1.0072629\\
&\geq \ln(2)+\ln (2 p_{q_{k_2-1}})+\ln(\ln (2 p_{q_{k_2-1}}))-1.0072629\\
& \geq \ln (2 p_{q_{k_2-1}})+\ln(\ln (2 p_{q_{k_2-1}}))-0.5\\
&=\ln (q_{k_2})+\ln(\ln (q_{k_2}))-0.5
\end{align*}
since $\ln(2)- 1.0072629=-0.314116\ldots$. Thus, we have
\begin{align*}
2p_{q_{k_2-1}}&\Big( \ln (p_{q_{k_1}}\cdot p_{q_{k_2-1}})+\ln(\ln (p_{q_{k_1}}\cdot p_{q_{k_2-1}}))-1.0072629\Big)\\
&\geq q_{k_2} \Big( \ln (q_{k_2})+\ln(\ln (q_{k_2}))-0.5\Big)\\
& \geq p_{q_{k_2}}\,,
\end{align*}
where the inequality in the last step follows from Theorem~\ref{RosSchoen} since $k_2\geq 4$ implies that $q_{k_2}\geq 20$.

It follows that
\begin{align*}
2p_{q_{k_1}} \cdot p_{q_{k_2-1}} \Big( \ln (p_{q_{k_1}}\cdot p_{q_{k_2-1}})&+\ln(\ln (p_{q_{k_1}}\cdot p_{q_{k_2-1}}))-1.0072629\Big)\\
& \geq p_{q_{k_1}} \cdot p_{q_{k_2}}\,,
\end{align*}
so that Theorem~\ref{Robin} together with the induction hypothesis gives us
\begin{align*}
q_k=2p_{q_{k-1}}&\geq 2 q_{k-1} \Big( \ln (q_{k-1})+\ln(\ln (q_{k-1}))-1.0072629\Big)\\
&\geq 2 p_{q_{k_1}} \cdot p_{q_{k_2-1}} \\
& \hspace*{1cm} \Big( \ln (p_{q_{k_1}} \cdot p_{q_{k_2-1}})+\ln(\ln (p_{q_{k_1}} \cdot p_{q_{k_2-1}}))-1.0072629\Big)\\
&\geq  p_{q_{k_1}} \cdot p_{q_{k_2}}
\end{align*}
since $k_1+k_2=k$ and $k> 6$ implies that $q_{k-1} \geq 2$. Therefore, we obtain
\begin{align*}
q_k=q_{k_1+k_2}\geq p_{q_{k_1}} \cdot p_{q_{k_2}} 
\end{align*}
for all pairs $(k_1,k_2)$ of positive integers. This completes the induction step as well as the proof of the lemma.
\end{proof}

We can now state and prove our next theorem:

\begin{theorem}
Among all topological trees with $n$ leaves, the binary caterpillar $F_n$ has the greatest Matula number. Furthermore, $M(F_n)=q_n$, where $q_n$ is the positive integer defined in Lemma~\ref{lemforqk}.
\end{theorem}

\begin{proof}
First of all, note that for $k>1$, we have $M(F_k)=p_1 \cdot p_{M(F_{k-1})}$ and by iteration on $k$, we obtain $M(F_k)=q_k$ for every $k$.

It is shown in Section~\ref{Towards} that the tree that maximises the Matula number among all topological trees with a prescribed number of leaves must be a binary tree. Now, let us use induction on $k$ to prove that $M(B)\leq M(F_k)$ for every binary tree $B$ with $k$ leaves.

The statement of the theorem is trivial for $k\leq 3$ since in this case, there is only one possibility for the shape of the binary tree. Assume that the statement holds for all binary trees with at most $k-1 \geq 3$ leaves and consider a binary tree $B$ with $k$ leaves. Denote by $B_1$ and $B_2$ the two branches of $B$. Since $M(B)=p_{M(B_1)}\cdot p_{M(B_2)}$, we obtain $M(B)\leq p_{M(F_{|B_1|})}\cdot p_{M(F_{|B_2|})}$ by the induction hypothesis. Consequently, invoking Lemma~\ref{lemforqk}, one obtains
\begin{align*}
M(B)\leq p_{M(F_{|B_1|})}\cdot p_{M(F_{|B_2|})} \leq M(F_{|B|})
\end{align*}
since $M(F_k)=q_k$. The theorem follows.
\end{proof}

\section{Conclusion}
The fact that the binary caterpillar $F_n$ is maximal in the set of all $n$-leaf topological trees with respect to the Matula number, also shows that $F_n$ is maximal in the set of all $n$-leaf binary trees with respect to the Matula number. Probably, the next natural problem for a future study will be to characterise the tree that minimises the Matula number over all binary trees (rooted trees in which every vertex has outdegree exactly $0$ or $2$) with a given number of leaves. 

Given a positive integer $k>1$, let $s$ be the unique nonnegative integer satisfying $2^{s+1}\leq k < 2^{s+2}$. Then write $k=r+2^{1+s}$ with $r$ the residue of $k$ modulo $2^{1+s}$. Define the sequence $(l_k)_{k\geq 1}$ of positive integers recursively by
\begin{align*}
l_k=l_{r+2^{1+s}}=\left\{
\begin{array}{rcl}
p_{l_{2^s}}\cdot p_{l_{r+2^s}}& \mbox{if} & r\leq 2^s \\ 
p_{l_r}\cdot p_{l_{2^{1+s}}} & \mbox{if} & r > 2^s\,,
\end{array}\right.
\end{align*}
starting with $l_1=1$. The sequence begins
\begin{align*}
1,4,14,49,301,1589,9761,51529,452411,3041573, 23140153,\ldots\,.
\end{align*}

Calculations suggest that $M(B)\geq l_k$ for every binary tree $B$ with $k\leq 18$ leaves. The $k$-leaf minimal binary tree $S_k$ in this case is also obtained in the same recursive way: the branches of $S_k=S_{r+2^{1+s}}$ are the binary trees $S_{2^s}$ and $S_{r+2^s}$ if $r\leq 2^s$, and the binary trees $S_r$ and $S_{2^{1+s}}$ if $ r > 2^s$, the starting tree being the tree $K_1$. For example, we show in Figure~\ref{MinBin} the minimal binary trees $S_6$ and $S_{13}$.
\begin{figure}[!h]\centering
\begin{subfigure}[b]{0.5\textwidth} \centering 	
	\begin{tikzpicture}
	\draw [fill=black] (0,0) circle (2.5 pt);
	\draw [fill=black] (-1,-1) circle (1.5 pt);
	\draw [fill=black] (1,-1) circle (1.5 pt);
	
    \draw [fill=black] (-1.5,-2) circle (1.5 pt);
    \draw [fill=black] (-0.5,-2) circle (1.5 pt);
	
	\draw [fill=black] (0.5,-2) circle (1.5 pt);
	\draw [fill=black] (1.5,-2) circle (1.5 pt);
	
	\draw [fill=black] (0.2,-2.5) circle (1.5 pt);
	\draw [fill=black] (0.8,-2.5) circle (1.5 pt);
	
	\draw [fill=black] (1.2,-2.5) circle (1.5 pt);
	\draw [fill=black] (1.8,-2.5) circle (1.5 pt);
	%%%%%%%%%%%%
	
	\draw [line width=1pt] (0,0) -- (-1,-1);
	\draw [line width=1pt] (0,0) -- (1,-1);
	
	\draw [line width=1pt] (-1,-1) -- (-1.5,-2);
	\draw [line width=1pt] (-1,-1)-- (-0.5,-2);
	
	\draw [line width=1pt] (1,-1) -- (0.5,-2);
	\draw [line width=1pt] (1,-1)-- (1.5,-2);
	
	\draw [line width=1pt] (0.5,-2) -- (0.2,-2.5);
	\draw [line width=1pt] (0.5,-2)-- (0.8,-2.5);
	
	\draw [line width=1pt] (1.5,-2) -- (1.2,-2.5);
	\draw [line width=1pt] (1.5,-2)-- (1.8,-2.5);
\end{tikzpicture}
\caption{$S_6$}
\end{subfigure} \vspace*{1cm}

%%%%%%%%%%%%%%
\begin{subfigure}[b]{1\textwidth} \centering 
\begin{tikzpicture}[scale=0.85]	
\draw [fill=black] (0,0) circle (2.5 pt);
\draw [fill=black] (-3,-1) circle (1.5 pt);
\draw [fill=black] (3,-1) circle (1.5 pt);
\draw [line width=1pt] (0,0) -- (-3,-1);
\draw [line width=1pt] (0,0) -- (3,-1);

\draw [fill=black] (-5,-2) circle (1.5 pt);	
\draw [fill=black] (-1,-2) circle (1.5 pt);	
\draw [line width=1pt] (-3,-1) -- (-5,-2);
\draw [line width=1pt] (-3,-1) -- (-1,-2);

\draw [fill=black] (1,-2) circle (1.5 pt);
\draw [fill=black] (5,-2) circle (1.5 pt);
\draw [line width=1pt] (3,-1) -- (1,-2);
\draw [line width=1pt] (3,-1) -- (5,-2);

\draw [fill=black] (-5.8,-3) circle (1.5 pt);	
\draw [fill=black] (-4.2,-3) circle (1.5 pt);
\draw [line width=1pt] (-5,-2) -- (-5.8,-3);
\draw [line width=1pt] (-5,-2) -- (-4.2,-3);

\draw [fill=black] (-1.8,-3) circle (1.5 pt);	
\draw [fill=black] (-0.2,-3) circle (1.5 pt);	
\draw [line width=1pt] (-1,-2) -- (-1.8,-3);
\draw [line width=1pt] (-1,-2) -- (-0.2,-3);

\draw [fill=black] (0.2,-3) circle (1.5 pt);
\draw [fill=black] (1.8,-3) circle (1.5 pt);
\draw [line width=1pt] (1,-2) -- (0.2,-3);
\draw [line width=1pt] (1,-2) -- (1.8,-3);

\draw [fill=black] (4.2,-3) circle (1.5 pt);
\draw [fill=black] (5.8,-3) circle (1.5 pt);
\draw [line width=1pt] (5,-2) -- (4.2,-3);
\draw [line width=1pt] (5,-2) -- (5.8,-3);

\draw [fill=black] (-6.3,-4) circle (1.5 pt);
\draw [fill=black] (-5.3,-4) circle (1.5 pt);
\draw [line width=1pt] (-5.8,-3) -- (-6.3,-4);
\draw [line width=1pt] (-5.8,-3) -- (-5.3,-4);

\draw [fill=black] (-4.7,-4) circle (1.5 pt);
\draw [fill=black] (-3.7,-4) circle (1.5 pt);
\draw [line width=1pt] (-4.2,-3) -- (-4.7,-4);
\draw [line width=1pt] (-4.2,-3) -- (-3.7,-4);

\draw [fill=black] (-2.3,-4) circle (1.5 pt);
\draw [fill=black] (-1.3,-4) circle (1.5 pt);
\draw [line width=1pt] (-1.8,-3) -- (-2.3,-4);
\draw [line width=1pt] (-1.8,-3) -- (-1.3,-4);

\draw [fill=black] (-0.7,-4) circle (1.5 pt);
\draw [fill=black] (0.3,-4) circle (1.5 pt);
\draw [line width=1pt] (-0.2,-3) -- (-0.7,-4);
\draw [line width=1pt] (-0.2,-3) -- (0.3,-4);

\draw [fill=black] (1.3,-4) circle (1.5 pt);
\draw [fill=black] (2.3,-4) circle (1.5 pt);
\draw [line width=1pt] (1.8,-3) -- (1.3,-4);
\draw [line width=1pt] (1.8,-3) -- (2.3,-4);
\end{tikzpicture}
\caption{$S_{13}$}
\end{subfigure} 

\caption{The minimal binary trees $S_6$ and $S_{13}$.}\label{MinBin}	
\end{figure}
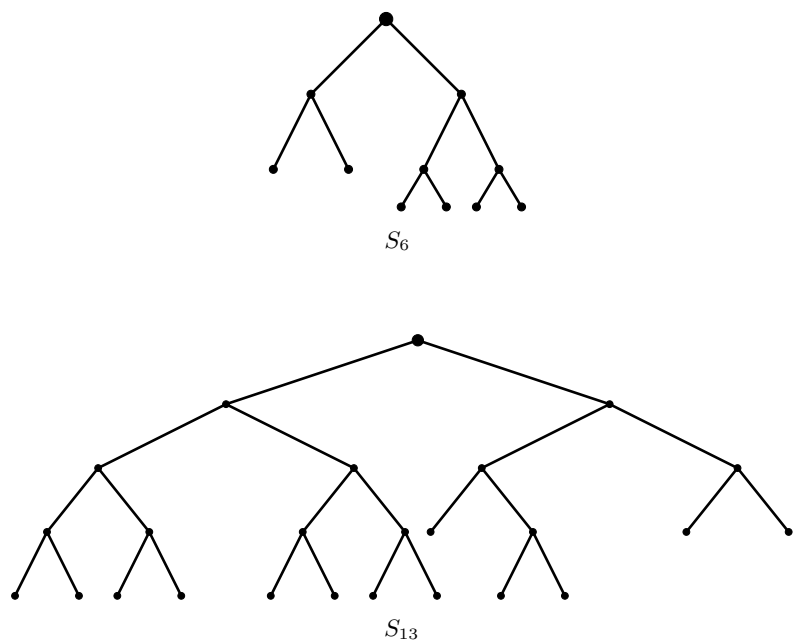

Note that the numbers $l_k$ soon get really large; for instance, the Matula number of $S_{18}$ is $32078140605053$. So it is difficult to check our evidence for more values of $k$; although
\begin{align*}
1.07555\times 10^{15} \leq p_{32078140605053} \leq 1.09182\times 10^{15}\,,
\end{align*}
a `standard' computer already fails to produce the prime number $$p_{32078140605053}.$$

On the other hand, since ``weak'' binary trees (every vertex has degree at most $3$) are more realistic molecular graphs, the problem of finding the extremal tree structures among these trees, given the number of vertices, may also be of interest. Recall that for the more general case of rooted trees, the problem has already been solved by Gutman and Ivi{\'c}~\cite{gutman1996matula}.

\section{Acknowledgements}
The author was supported by Stellenbosch University in association with the African Institute for Mathematical Sciences (AIMS) South Africa.

\end{document}